\documentclass[a4paper,12pt,reqno,twoside]{amsart}

\usepackage[utf8]{inputenc} 		% UTF8 encoding
\usepackage[T1]{fontenc} 

\usepackage{amssymb}
\usepackage{amsmath}
\usepackage{graphicx}
\usepackage[colorlinks,citecolor=blue,urlcolor=blue]{hyperref}
\usepackage{cite}
\usepackage[hmargin=2.5cm,vmargin=2.5cm]{geometry}
\usepackage{mathrsfs} % Script fonts

\newcommand{\be}{\begin{eqnarray}}
\newcommand{\ee}{\end{eqnarray}}

\newcommand{\beq}{\begin{equation}}
\newcommand{\eeq}{\end{equation}}
\newcommand{\beqn}{\begin{equation*}}
\newcommand{\eeqn}{\end{equation*}}

\newcommand{\slot}{\,\cdot\,}
\newcommand{\round}[1]{\lfloor#1\rfloor}

\newcommand{\ave}[1]{{\langle#1\rangle}}

\newtheorem{thm}{Theorem}[section]

\newtheorem{cor}[thm]{Corollary}
\newtheorem{lem}[thm]{Lemma}
\newtheorem{defn}[thm]{Definition}

\newtheorem{remark}[thm]{Remark}

\newcommand\cF{{\mathcal F}}

\newcommand\cK{{\mathcal K}}

\newcommand\cM{{\mathcal M}}

\newcommand\cT{{\mathcal T}}

\newcommand\bE{{\mathbb E}}

\newcommand\bN{{\mathbb N}}

\newcommand\bP{{\mathbb P}}

\newcommand\bR{{\mathbb R}}
\newcommand\bS{{\mathbb S}}

\newcommand\rd{{\mathrm d}}

\newcommand\fm{{\mathfrak m}}

\newcommand{\ve}{\varepsilon}

\def\bfT{\mathbf{T}}

\begin{document}

\title{An almost sure ergodic theorem for quasistatic dynamical systems}

\author[Mikko Stenlund]{Mikko Stenlund}
\address[Mikko Stenlund]{
Department of Mathematics and Statistics, P.O.\ Box 68, Fin-00014 University of Helsinki, Finland.}
\email{mikko.stenlund@helsinki.fi}
\urladdr{http://www.math.helsinki.fi/mathphys/mikko.html}


 % Suggesting these. 
 
% http://www.ams.org/msc/msc2010.html
% 37D20  	Uniformly hyperbolic systems (expanding, Anosov, Axiom A, etc.)
% 37C40  	Smooth ergodic theory, invariant measures
% 37C60  	Nonautonomous smooth dynamical systems
% 60G44  Martingales with continuous parameter
% 60H10  	Stochastic ordinary differential equations

%\date{\today. {\bf Please do not circulate!}}

\begin{abstract}
We prove almost sure ergodic theorems for a class of systems called quasistatic dynamical systems. These results are needed, because the usual theorem due to Birkhoff does not apply in the absence of invariant measures. We also introduce the concept of a physical family of measures for a quasistatic dynamical system. These objects manifest themselves, for instance, in numerical experiments. We then verify the conditions of the theorems and identify physical families of measures for two concrete models, quasistatic expanding systems and quasistatic dispersing billiards.

\end{abstract}

\maketitle

%%%%%%%%%%%%%%%%%%%%%%%%%%%%
%%%%%%%%%%%%%%%%%%%%%%%%%%%%

\subsection*{Acknowledgements}
This work was supported by Jane and Aatos Erkko Foundation, and by Emil Aaltosen S\"a\"ati\"o. The author would like to thank Ville Hakulinen for a useful conversation.

%%%%%%%%%%%%%%%%%%%%%%%%%%%%%%%%%%%%
%%%%%%%%%%%%%%%%%%%%%%%%%%%%%%%%%%%%
%%%%%%%%%%%%%%%%%%%%%%%%%%%%%%%%%%%%

\medskip
\section{Introduction}
This paper is about ergodic properties of quasistatic dynamical systems. These are generalizations of conventional dynamical systems, and belong to the realm of deterministic non-equilibrium processes. Such systems, which suffer from the lack of invariant measures, have become a timely research topic due to the recent progress in the study of time-dependent dynamical systems; see~for example\cite{OttStenlundYoung_2009,Stenlund_2011,StenlundYoungZhang_2013,GuptaOttTorok_2013,MohapatraOtt_2014,Aimino_etal_2014}. We define the concept in Section~\ref{sec:QDS} and also describe the problem addressed in this paper. Then, in Section~\ref{sec:families}, we introduce what we call physical families of measures. The latter is a new concept which is particularly suited for making statistical predictions about the long-term behavior of quasistatic dynamical systems.

%%%%%%%%%%%%%%%%%%%%%%%%%%%%%%%%%%%%%%%

\medskip
\subsection{Notations and conventions}
Given an integer~$u\ge 0$,
\beqn
\round{u} = \max\{k\in\bN\,:\,k\le u\}
\eeqn
is its integer part.
All measures --- denoted by $\mu$, $\fm$, $\hat\mu$, $\hat\mu_t$, etc --- are probability measures.
All functions~$f:X\to\bR$ and self-maps $T:X\to X$ on a measurable space~$(X,\mathscr{F})$ are measurable. Given a function~$f$ and a measure~$\mu$, we write~$\mu(f)$ for~$\int_X f\,\rd\mu$. All integrals over~$X$ and over an interval~$I\subset [0,1]$ are to be understood in the sense of Lebesgue.

%%%%%%%%%%%%%%%%%%%%%%%%%%%%%%%%%%%%%%%

\medskip
\subsection{Quasistatic dynamical systems}\label{sec:QDS}

We now recall the notion of a quasistatic dynamical system, which was introduced in~\cite{DobbsStenlund_2014}. In this paper we restrict to the measurable case because of the type of result we are after. We first give the mathematical definition and then discuss its physical meaning.

\medskip
\begin{defn}\label{defn:QDS}
Let~$(X,\mathscr{F})$ be a measurable space, $\cM$~a topological space whose elements are measurable self-maps $T:X\to X$, and~$\bfT$ a triangular array of the form
\beq\label{eq:array}
\bfT = \{T_{n,k}\in\cM\ :\ 0\le k\le n, \ n\ge 1\} .
\eeq
If there exists a piecewise-continuous curve $\gamma:[0,1]\to\cM$ such that
\beq\label{eq:gamma_limit}
\lim_{n\to\infty}T_{n,\round{nt}} = \gamma_t
\eeq
we say that $(\bfT, \gamma)$ is a (measurable) \emph{quasistatic dynamical system (QDS)} with \emph{state space}~$X$ and \emph{system space}~$\cM$. 
\end{defn}

\medskip
\begin{defn}\label{defn:associated_measure}
It often happens that, given $t\in[0,1]$, the map~$\gamma_t$ will have an invariant probability measure of special interest, the choice of which depends on the question one is posing. We denote such a designated measure by~$\hat\mu_t$ and call it \emph{the measure associated to~$\gamma_t$}.
\end{defn}

\medskip
Given a QDS, the description of the dynamics is as follows: A point~$x$ of the state space~$X$ represents the state of the system. The evolution of~$x$ is given by the triangular array~$\bfT$, separately on each level of the array. That is, given~$n\ge 1$, the point 
\beqn
x_{n,k} = T_{n,k}\circ\dots\circ T_{n,1}(x) \in X
\eeqn
is the state of the system after~$k\in\{0,\dots,n\}$ steps on the~$n$th level of the array~$\bfT$. The convention here is that $x_{n,0} = x$. Our objective is to study the statistical properties of~$(x_{n,k})_{k=0}^n$ in the limit $n\to\infty$. 

We point out that a conventional dynamical system $T:X\to X$ is a fully degenerate instance of a QDS, corresponding to $T_{n,k}=T$ for all $k$ and $n$, in which case $x_{n,k}=T^k(x)$.

Introducing the continuous parameter $t\in[0,1]$ and setting $k = \round{nt}$, the piecewise-constant curve~$t\mapsto T_{n,\round{nt}}$ approximates~$\gamma$ with ever-increasing accuracy in~$\cM$, as $n\to\infty$. It is helpful to think of~$t$ as macroscopic time and, with $n$ fixed, of $k = \round{nt}$ as the corresponding microscopic time. In terms of macroscopic time, we are interested in the statistical properties of~$(x_{n,\round{nt}})_{0\le t\le 1}$ as $n\to\infty$. In the analysis of concrete systems, as is the case with the systems considered in Section~\ref{sec:applications}, the rate of convergence of $T_{n, \round{nt}}$ to~$\gamma_t$ may need to be specified.

Let us now briefly explain the physical motivation behind the definitions. In thermodynamics, the adjective quasistatic is attributed to an idealized process where external influences cause the observed system to transform infinitesimally slowly; the system remains in thermodynamic equilibrium at any instant, but the equilibrium changes slowly over an infinitely long time span. Examples of quasistatic processes include all reversible ones. 
Analogously, a QDS models a situation in which the dynamical system changes infinitesimally slowly due to external influence. This is expressed by~\eqref{eq:gamma_limit}. As measured in microscopic time $k=\round{nt}$, the system~$T_{n,k} \approx \gamma_t$ (roughly) traverses the curve~$\gamma$ from beginning to end, and this happens more and more slowly as~$n$ increases. Despite the slowness of the changes, in interesting situations the system may morph drastically over a sufficiently long microscopic timescale: unless $\gamma$ is a constant curve,~$T_{n,n}\approx \gamma_1$ differs substantially from~$T_{n,1}\approx \gamma_0$ for all large values of~$n$. The instantaneous equilibria of the quasistatic system are described by the invariant measures $\hat\mu_t$ associated to the limit maps $\gamma_t$, $0\le t\le 1$. We have included the possibility of having jumps in~$\gamma$ in order to accommodate for larger, but rare, abrupt changes in the system as well. For instance, these may be jumps from one connected component of the system space~$\cM$ to another.

Recall that~$f$ will always stand for a measurable function $f:X\to\bR$.
We denote
\beq\label{eq:f}
f_{n,k} = f\circ T_{n,k}\circ\dots\circ T_{n,1}, \quad 0\le k \le n.
\eeq
The convention is that $f_{n,0} = f$. We define the functions $S_n:X\times [0,1]\to\bR$ by
\beqn
S_n(x,t) = \int_0^{nt} f_{n,\round{s}}(x)\,\rd s, \quad n\ge 1.
\eeqn
Note that, given~$x$, the function~$S_n(x,\slot)$ is a piecewise linear interpolation of the Birkhoff-type sums~$\sum_{k=0}^{\round{nt}-1}  f_{n,k}(x)$ and, as such, an element of~$C([0,1])$ --- the space of continuous functions from~$[0,1]$ to~$\bR$.

A reasonable question to ask is whether
\beq\label{eq:zeta}
\zeta_n(x,t) = n^{-1} S_n(x,t) = \int_0^{t} f_{n,\round{ns}}(x)\,\rd s
\eeq
converges in some sense. Given an initial distribution~$\mu$ for~$x$, one may view~$\zeta_n$ as a random element of~$C([0,1])$. We equip $C([0,1])$ with the uniform norm and the Borel sigma-algebra.

Note that in the fully degenerate case --- $T_{n,k}=T$ for all~$k$ and~$n$ --- Birkhoff's ergodic theorem guarantees that, given a bounded measurable $f$, $\lim_{n\to\infty}\zeta_n(x,t) = t\int_X f \,\rd\mu$ for almost every~$x$ with respect to an ergodic~$T$-invariant measure~$\mu$, for all~$t\in[0,1]$.

Taking the random viewpoint above, a special class of quasistatic dynamical systems was studied in~\cite{DobbsStenlund_2014}: the system space $\cM$ of the model consists of smooth expanding maps on the circle~$X=\bS$ and the curve $\gamma$ is piecewise H\"older continuous;  see Section~\ref{sec:expanding} below for details. Each map $\gamma_t$ has a unique ergodic Sinai--Ruelle--Bowen (SRB) measure~$\hat\mu_t$ associated to it. Define~$\zeta\in C([0,1])$ by the expression
\beq\label{eq:zeta_limit}
\zeta(t) = \int_0^t \hat\mu_s(f)\,\rd s.
\eeq 
Under the assumptions that~$\mu$ be absolutely continuous and~$f$ be Lipschitz continuous, it was shown in~\cite{DobbsStenlund_2014} that
\beq\label{eq:as}
\lim_{n\to\infty}\zeta_n = \zeta
\eeq
\emph{in distribution}. That is, the law of~$\zeta_n$ on the measurable space $C([0,1])$ converges weakly to the Dirac measure at~$\zeta$.

In this paper we obtain a stronger mode of convergence. We identify conditions for quasistatic dynamical systems under which~\eqref{eq:as} is true \emph{almost surely} with respect to a distribution~$\mu$. We then apply the result to specific models in Section~\ref{sec:applications}, including the one studied in~\cite{DobbsStenlund_2014}.

%%%%%%%%%%%%%%%%%%%%%%%%%%%%%%%%%%%%%%%

\medskip
\subsection{Physical families of measures}\label{sec:families}
Let $X$ be a topological space equipped with the Borel sigma-algebra, and let $\fm$ be a reference measure on $X$. The picture to have in mind is that~$\fm$ is associated to~$X$ in a canonical way; for example it could be the phase-space volume. Let $T:X\to X$ be a measurable map with an invariant measure $\hat\mu$. We do \emph{not} require~$\fm$ to be invariant for~$T$.

\pagebreak
Let us recall the concept of a physical measure; see, e.g.,~\cite{Young_2002}.

\medskip
\begin{defn}
Suppose there exists a measurable set $A\subset X$ with $\fm(A)>0$ such that
\beq\label{eq:physical}
\lim_{n\to\infty}n^{-1}\sum_{k=0}^{n-1} f\circ T^k(x) = \hat\mu(f)
\eeq
holds for all $x\in A$ and all bounded continuous functions $f:X\to\bR$. Then we say that~$\hat\mu$ is a \emph{physical measure} for~$T$. Any point $x\in X$ satisfying \eqref{eq:physical} for all bounded continuous~$f$ is called~$\hat\mu$-generic.
\end{defn}

\medskip
We have required $f$ to be bounded in order to ignore integrability issues. In truth, the above definition is an abstract version of the standard one, where $X$ is a compact oriented Riemannian manifold and~$\fm$ is the measure determined by the volume form.

A physical measure~$\hat\mu$ captures the average behavior of $f\circ T^k(x)$, $k\ge 0$, for a significant fraction of initial points~$x$ \emph{in the sense of the reference measure $\fm$}. For instance in a computer experiment the point~$x$ could be drawn from a uniform distribution, and one might observe that, for a positive fraction of such points, the time-average on the left side of~\eqref{eq:physical} tends to the expected value on the right side. Hence~$\hat\mu$ would stand out among all invariant measures as one that is relevant to physical observations. A prototypical example of a physical measure~$\hat \mu$ is the SRB measure on a hyperbolic (axiom~A) attractor of a diffeomorphism~$T$, in which case $A$ is the basin of attraction and $\fm$ is the phase-space volume.

Of course, by Birkhoff's theorem, if~$\fm$ happens to be an ergodic invariant measure for~$T$, then~\eqref{eq:physical} holds for~$\hat\mu=\fm$ and $A\subset X$ with $\fm(A)=1$. In many applications, however, the canonical choice for~$\fm$ will not be invariant, let alone ergodic, as hinted above.

In the context of quasistatic dynamical systems, the notion of a physical measure is generally unnatural. Instead, the following notion of a physical family of measures will turn out to be useful.

\medskip
\begin{defn}\label{defn:physical_family}
Let $\bfT$ be a triangular array as in~\eqref{eq:array}. Let $\mathscr{P} = (\hat\mu_t)_{t\in[0,1]}$ be a one-parameter family of measures on~$X$, and suppose that the map $t\mapsto\hat\mu_s(f)$ is measurable for all bounded continuous functions~$f:X\to\bR$. Suppose there exists a measurable set~$A\subset X$ with $\fm(A)>0$ such that %$\lim_{n\to\infty}\zeta_n(x,t) = \zeta(t)$, i.e.,
\beq\label{eq:physical_QDS}
\lim_{n\to\infty} \int_0^{t} f_{n,\round{ns}}(x)\,\rd s = \int_0^t \hat\mu_s(f)\,\rd s, \quad t\in[0,1],
\eeq
holds for all $x\in A$ and all bounded continuous functions~$f:X\to\bR$. Then we say that~$\mathscr{P}$ is a \emph{physical family of measures} for~$\bfT$. Any point $x\in X$ satisfying \eqref{eq:physical_QDS} for all bounded continuous $f$ is called~$\mathscr{P}$-generic.

\medskip
\end{defn}
We do not distinguish between two physical families of measures $(\hat\mu_t)_{t\in[0,1]}$ and $(\hat\mu'_t)_{t\in[0,1]}$ in case $\hat\mu_t = \hat\mu'_t$ for all but a zero-measure set of parameters~$t$.

The importance of physical families of measures is similar to that of physical measures in the case of a conventional dynamical system.

In Section~\ref{sec:applications} we will identify physical families of measures for concrete quasistatic dynamical systems.

%%%%%%%%%%%%%%%%%%%%%%%%%%%%%%%%%%%%%%%

\medskip
\subsection{Structure of the paper}
In Section~\ref{sec:abstract} we present almost sure ergodic theorems for abstract quasistatic dynamical systems. In that section we also present a theorem on the uniqueness of a physical family of measures. In Section~\ref{sec:applications} we give applications of the latter to concrete quasistatic dynamical systems. Together, these sections comprise the main results of the paper. In Section~\ref{sec:preliminary} we prove preliminary lemmas necessary for the proofs of the abstract theorems. The theorems are then proven in Section~\ref{sec:proofs}.

%%%%%%%%%%%%%%%%%%%%%%%%%%%%%%%%%%%%%%%
%%%%%%%%%%%%%%%%%%%%%%%%%%%%%%%%%%%%%%%

\medskip
\section{Main theorems}\label{sec:abstract}

\medskip
\subsection{Ergodic theorems}
Let us introduce the centered quantity
\beqn
\bar\zeta_n(x,t) = \zeta_n(x,t)-\mu(\zeta_n(\slot,t)).
\eeqn
Note that
\beqn
\bar\zeta_n(x,t) = \int_0^t \bar f_{n,\round{ns}}(x)\,\rd s,
\eeqn
where
\beqn
\bar f_{n,\round{ns}} = f_{n,\round{ns}} - \mu(f_{n,\round{ns}}).
\eeqn

\medskip
We begin with a result for triangular arrays~$\bfT$ which does not presuppose convergence to a curve~$\gamma$. The proof is given in Section~\ref{sec:proofs}.

\medskip
\begin{thm}\label{thm:main1}
Let~$f$ be a bounded function and~$\mu$ a probability measure. Suppose the following conditions hold:
\medskip
\begin{itemize}
\item[(A1)]
There exist a dense set $A\subset[0,1]$ and a function $g:A\to\bR$ such that
\beqn
\lim_{t\to\infty}\mu(\zeta_n(\slot,t)) = g(t), \quad t\in A ; 
\eeqn
\smallskip
\item[(A2)]
There exist a dense set $B\subset[0,1]$ and a number $p\ge 1$ such that
\beqn
\sum_{n=1}^\infty \mu(|\bar\zeta_n(\slot,t)|^p) < \infty, \quad t\in B.
\eeqn
\end{itemize}
\medskip
Then $g$ extends to a Lipschitz continuous function on~$[0,1]$ (denoted by~$g$) such that
\beqn%\label{eq:uniform_conv}
\lim_{n\to\infty}\sup_{t\in[0,1]}|\zeta_n(x,t)-g(t)| = 0
\eeqn
for almost every $x$ with respect to $\mu$.
\end{thm}

\medskip
Condition (A1) asks for nice asymptotic behavior of the mean
\beqn%\label{eq:zeta}
\mu(\zeta_n(\slot,t)) = \int_0^{t} \mu(f_{n,\round{ns}})\,\rd s,
\eeqn
whereas (A2) is a moment condition implying almost sure convergence of $\bar\zeta_n(\slot,t)$ to~$0$.
The role of the sets~$A$ and~$B$ in the theorem is that in some situations it is convenient to restrict to a dense set, such as the one of dyadic rational numbers, instead of the entire interval~$[0,1]$.

\pagebreak
\medskip
In practice, $g=\zeta$ is natural in the context of QDSs. To get a heuristic idea of why this is so, observe first that~\eqref{eq:f} yields $\mu(f_{n,\round{ns}}) = (T_{n,\round{ns}})_*\cdots (T_{n,1})_*\mu(f)$, where $T_*\mu$ denotes the pushforward measure $T_*\mu(A) = \mu(T^{-1}A)$, $A\in\mathscr{F}$. On the other hand, the limit in~\eqref{eq:gamma_limit} and the (piecewise) continuity of the curve~$\gamma$ suggest that
\beqn
\begin{split}
\mu(f_{n,\round{ns}}) & = (T_{n,\round{ns}})_*\cdots (T_{n,\round{n(s-\ve)}+1})_* (T_{n,\round{n(s-\ve)}})_*\cdots (T_{n,1})_*\mu(f)
\\
& \approx (\gamma_s)_*^{\round{n\ve}} (T_{n,\round{n(s-\ve)}})_*\cdots (T_{n,1})_*\mu(f)
\end{split}
\eeqn
for small $\ve>0$ and large~$n$. If $(T_{n,\round{n(s-\ve)}})_*\cdots (T_{n,1})_*\mu$ remains in a class of reasonable measures $\nu$ for all $n$, and if the maps $\gamma_s$ have a memory-loss property $\lim_{m\to\infty}(\gamma_s)_*^m\nu(f) = \hat\mu_s(f)$ for such measures, one can expect $\lim_{n\to\infty}\mu(f_{n,\round{ns}}) = \hat\mu_s(f)$, i.e., $g=\zeta$. Here $\hat\mu_s$ stands for the measure associated to~$\gamma_s$; see Definition~\ref{defn:associated_measure}.

The following result identifies conditions for the conclusion of the previous theorem with $g=\zeta$. We will check these conditions for specific models in Section~\ref{sec:applications}. 

In order to state the result, we need to introduce certain correlation functions. To that end, let~$f$ and~$\mu$ be a given function and a probability measure, respectively. We write 
\beq\label{eq:corr}
c^{\ell,j}_n(k_1,\dots,k_\ell) = \mu(f_{n,k_1}\cdots f_{n,k_\ell}) - \mu(f_{n,k_1}\cdots f_{n,k_j}) \, \mu(f_{n,k_{j+1}}\cdots f_{n,k_\ell})
\eeq
for all integers $\ell\ge 2$, $1\le j<\ell$ and $k_1,\dots,k_\ell\ge 0$. In this paper we will only be interested in $2\le \ell\le 4$ and $j\in\{1,\ell-1\}$. Note that if $c^{\ell,j}_n(k_1,\dots,k_\ell)$ is small, then the products $f_{n,k_1}\cdots f_{n,k_j}$ and $f_{n,k_{j+1}}\cdots f_{n,k_\ell}$ are nearly uncorrelated with respect to the initial distribution~$\mu$.  We also introduce the function~$\Phi:\bR_+\to\bR_+$ defined by
\beqn
\Phi(s) = 
\begin{cases}
s^{-1}(\log s)^{-2}, & s\ge 2,
\\
2^{-1}(\log 2)^{-2}, & 0\le s< 2.
\end{cases}
\eeqn
Note that $\Phi$ is integrable and non-increasing on $\bR_+$.

\medskip
\begin{thm}\label{thm:main2}
Let $(\bfT, \gamma)$ be a QDS and $\hat\mu_t$ the measure associated to $\gamma_t$ for each $t\in[0,1]$.
Let~$f$ be a bounded function and~$\mu$ a probability measure. Suppose the following conditions hold:
\smallskip
\begin{itemize}
\item[(B1)] The map $t\mapsto\hat\mu_t(f)$ is measurable;
\medskip
\item[(B2)] $\lim_{n\to\infty}\mu(f_{n,\round{nt}})=\hat\mu_t(f)$ for almost every $t\in[0,1]$;
\medskip
\item[(B3)] There exists $C>0$ such that
\beqn
|c^{\ell,j}_n(k_1,\dots,k_\ell)| \le C\Phi(k_{j+1}-k_j)
\eeqn
for all integers $2\le \ell \le 4$, $j\in\{1,\ell-1\}$ and $0\le k_1\le\dots\le k_\ell$. 
\end{itemize}
\medskip
Then
\beq\label{eq:uniform_conv}
\lim_{n\to\infty}\sup_{t\in[0,1]}|\zeta_n(x,t)-\zeta(t)| = 0
\eeq
for almost every $x$ with respect to $\mu$.
\end{thm}

\medskip
\begin{remark}
If $f$ is vector valued, $f:X\to\bR^d$, similar results hold: it suffices to check the conditions of Theorems~\ref{thm:main1} and~\ref{thm:main2} for each vector component separately.
\end{remark}

\medskip
Conditions (B1)--(B3) were discussed in the paragraphs preceding Theorem~\ref{thm:main2}. Let us mention, however, that the proof of the theorem amounts to showing that (B1)--(B2) imply (A1) with $g=\zeta$ and~(B3) implies (A2) with $p=4$; see Section~\ref{sec:proofs} for details.

%%%%%%%%%%%%%%%%%%%%%%%%%%%%%%%%%%%%

\medskip
\subsection{Uniqueness of a physical family of measures}
Consider a topological space~$X$ which is separable and metrizable. Then~$X$ is second-countable, i.e., its topology has a countable base, and the Borel sigma-algebra on~$X$ is countably generated.

\medskip
The following uniqueness theorem shows that, on such spaces, physical families of measures are characterized by their generic points (Definition~\ref{defn:physical_family}).

\medskip
\begin{thm}\label{thm:unique}
Suppose~$X$ is a separable and metrizable space equipped with the Borel sigma-algebra. Let $\bfT$ be a triangular array as in~\eqref{eq:array} and let the reference measure~$\fm$ be given. Let~$\mathscr{P}$ and~$\mathscr{P}'$ be two physical families of measures. If there exists a point $x\in X$ that is both~$\mathscr{P}$-generic and~$\mathscr{P}'$-generic, then $\mathscr{P} = \mathscr{P}'$. 
\end{thm}

\medskip
We prove Theorem~\ref{thm:unique} in Section~\ref{sec:unique_proof}.

\medskip
We also get the following corollary:

\medskip
\begin{cor}\label{cor:unique}
Let $X$, $\bfT$ and $\fm$ be as above. If $\mathscr{P}$ is a physical family of measures such that almost every point $x\in X$ with respect to~$\fm$ is $\mathscr{P}$-generic, then $\mathscr{P}$ is unique.
\end{cor}

%\medskip
\begin{proof}
Let~$\mathscr{P}'$ be a physical family of measures. Denote by~$A$ and~$A'$ the sets of $\mathscr{P}$-generic and $\mathscr{P}'$-generic points, respectively. Then $\fm(A) = 1$ and $\fm(A')>0$ by assumption. In particular, $A\cap A' \ne \varnothing$, so Theorem~\ref{thm:unique} implies that $\mathscr{P} = \mathscr{P}'$.
\end{proof}

%%%%%%%%%%%%%%%%%%%%%%%%%%%%%%%%%%%%%%%
%%%%%%%%%%%%%%%%%%%%%%%%%%%%%%%%%%%%%%%

%\pagebreak
\medskip
\section{Applications}\label{sec:applications}
In this section we put Theorem~\ref{thm:main2} to use and show that the ergodic property~\eqref{eq:uniform_conv} holds for concrete systems. 

\medskip
\subsection{Quasistatic expanding system}\label{sec:expanding}
In this section~$\bS$ will stand for the circle, obtained by identifying the endpoints of the unit interval, and~$\fm$ will stand for the Lebesgue measure on it. We denote by~$d$ the standard metric on~$\bS$.

Let us fix $\lambda>1$ and $A_*>0$, and 
let~$\cM$ denote the set of~$C^2$ maps~$T:\bS\to\bS$ satisfying
\beqn
\inf T' \ge \lambda
\quad\text{and}\quad
\| T''\|_\infty \le A_* .
\eeqn
We equip $\cM$ with the metric $d_{C^1}$ defined by 
\beqn
d_{C^1}(T_1,T_2) = \sup_{x\in\bS} d(T_1 x,T_2 x) + \|T_1'-T_2'\|_\infty.
\eeqn

Next, let $\gamma:[0,1]\to \cM$ be a piecewise H\"older continuous curve with exponent $\eta\in(0,1)$ and
$\bfT$ a triangular array
\beqn
\bfT= \{T_{n,k}\in\cM\ :\ 0\le k\le n, \ n\ge 1\}
\eeqn
such that, for some $C>0$,
\beqn \label{eq:rate}
\sup_{0\le t\le 1} d_{C^1}(T_{n,\round{nt}},\gamma_t) \le Cn^{-\eta}.
\eeqn
Then $(\bfT, \gamma)$ is a QDS with state space~$X=\bS$. 
The rate $n^{-\eta}$ is related to the regularity of the curve $\gamma$ in a natural way: it is the optimal rate in the canonical case $T_{n,k} = \gamma_{kn^{-1}}$.

We remind the reader that, for every $T\in\cM$, there exists a unique invariant probability measure~$\hat\mu_T$ which is equivalent to~$\fm$. Below, we will denote
\beqn
\hat\mu_{t} = \hat\mu_{\gamma_t}.
\eeqn

\medskip
The proof of the following lemma can be found in~\cite{DobbsStenlund_2014}:

\medskip
\begin{lem}\label{lem:QDS_expanding}
Let~$f$ be Lipschitz continuous and~$\mu = \fm$. The map $t\mapsto\hat\mu_t(f)$ is piecewise continuous. For any $\eta'\in(0,\eta)$ there exists $C>0$ such that
\beqn
\sup_{t\in[0,1]}|\mu(f_{n,\round{nt}})-\hat\mu_t(f)| \le C n^{-\eta'}.
\eeqn
There exist $D>0$ and $\vartheta\in(0,1)$ such that the correlation functions defined in~\eqref{eq:corr} satisfy
\beqn
|c^{\ell,j}_n(k_1,\dots,k_\ell)| \le D\vartheta^{k_{j+1}-k_j}
\eeqn
for all integers $2\le \ell \le 4$, $1\le j<\ell$ and $0\le k_1\le\dots\le k_\ell$. 
\end{lem}

\medskip
\begin{remark}\label{rem:Lip}
The Lipschitz continuity assumption in the lemma can be relaxed, but this is of little interest for our present purposes; see the result below.
\end{remark}

\medskip
We are now in position to prove the following ergodic theorem for the quasistatic expanding system:

\medskip
\begin{thm}\label{thm:QDS_expanding}
Let $f$ be continuous. Then
\beq\label{eq:uniform_conv_expanding}
\lim_{n\to\infty}\sup_{t\in[0,1]}|\zeta_n(x,t)-\zeta(t)| = 0
\eeq
for almost every $x$ in the sense of Lebesgue. In particular, $(\hat\mu_t)_{t\in[0,1]}$ is the unique physical family of measures for the quasistatic dynamical system in question.
\end{thm}

\medskip
\begin{proof}
Let $f$ be Lipschitz continuous and $\mu=\fm$. Lemma~\ref{lem:QDS_expanding} shows that the conditions of Theorem~\ref{thm:main2} are satisfied. Hence,~\eqref{eq:uniform_conv_expanding} holds true for Lipschitz continuous functions. We extend it to all continuous functions. To that end, we argue by approximation. We write $\zeta_n(f;x,t)$ for $\zeta_n(x,t)$ when it is necessary to emphasize which function~$f$ is used in the definition of $\zeta_n$.  Let $f$ be a continuous function.
Given $\ve>0$, there exists a Lipschitz continuous function $g$ such that
\beqn
\|f-g\|_\infty < \ve,
\eeqn
because Lipschitz continuous functions are dense in the space of continuous functions equipped with the uniform norm.
Then
\beqn
\sup_{t\in[0,1]} |\zeta_n(f;x,t) - \zeta_n(g;x,t)| \le \int_0^1 |f_{n,\round{ns}}(x) - g_{n,\round{ns}}(x)|\,\rd s < \ve.
\eeqn
for all~$x$ and~$n$. On the other hand, what was already proven above implies that there exists $\Omega_\ve$ with $\fm(\Omega_\ve)=1$ such that
\beqn%\label{eq:uniform_conv}
\lim_{n\to\infty}\sup_{t\in[0,1]}|\zeta_n(g;x,t)-\zeta(t)| = 0
\eeqn
for all~$x\in\Omega_\ve$, so
\beqn
\limsup_{n\to\infty}\sup_{t\in[0,1]}|\zeta_n(f;x,t)-\zeta(t)| < \ve, \quad x\in\Omega_\ve.
\eeqn
Now, set $\Omega_0 = \cap_{k\ge 1}\Omega_{1/k}$. Then $\fm(\Omega_0) = 1$ and
\beqn
\lim_{n\to\infty}\sup_{t\in[0,1]}|\zeta_n(f;x,t)-\zeta(t)| = 0, \quad x\in\Omega_0.
\eeqn
This is what was to be shown. Obviously $(\hat\mu_t)_{t\in[0,1]}$ is then a physical family of measures. Its uniqueness is a consequence of Corollary~\ref{cor:unique} since $\fm(\Omega_0) = 1$.
\end{proof}

%%%%%%%%%%%%%%%%%%%%%%%%%%%%%%%%%%

\medskip
\subsection{Quasistatic billiards}
Dispersing billiards on a 2-dimensional torus is an important model in which a point particle moves linearly on the surface of the torus among \emph{fixed}, strictly convex, scatterers with smooth boundaries. Upon collision with a scatterer, the particle bounces off elastically and continues its linear motion. If the scatterers are placed in such a way that the length of the free path between any two successive collisions is bounded, recording collision points gives a rather accurate representation of the position of the particle as a function of time. In other words, given the scatterer configuration~$K$, the dynamics of the particle is represented by a billiard map $F_{K}$ mapping one collision to the next. A standard textbook on billiards is~\cite{ChernovMarkarian_2006}.

Following~\cite{StenlundYoungZhang_2013}, we now discuss dispersing billiards with \emph{moving} scatterers.
To that end, let $\cK$ be the collection of admissible scatterer configurations. By admissibility we mean that the strictly convex scatterers have sufficiently smooth boundaries with uniform curvature bounds, and that the free path length is uniformly bounded; see~\cite{StenlundYoungZhang_2013} for precise conditions. There is a natural distance on~$\cK$, which we denote by~$d$. There exist~$\ve_0>0$ and~$X$ (called the ``collision space'') such that, given any pair $(K,K')\in \cK\times\cK$ satisfying~$d(K,K')<\ve_0$, one can define a map $F_{K',K}:X\to X$ mapping a collision with a scatterer~$S$ in the source configuration~$K$ to the next collision with a scatterer~$S'$ in the target configuration~$K'$. The physical interpretation is that, after the particle collided with~$S$, the scatterers had a short time to move by a small amount ($\ve_0$ is \emph{very small}) to a new configuration~$K'$, before the particle finally met~$S'$. Composing such maps one can model the motion of the particle amongst slowly moving scatterers (the speed being bounded by $\ve_0$).

%A model of dispersing billiards with moving scatterers was introduced in~\cite{StenlundYoungZhang_2013}: Let~$(K_k)_{k=0}^\infty$ be a sequence of scatterer configurations, such that $d(K_{k-1},K_k)<\ve$ holds uniformly for some small~$\ve>0$. Here~$d$ is a natural distance on the space~$\cK$ of admissible configurations.
%Then $F_{K_k,K_{k-1}}$ represents the dynamics between the~$(k-1)$th and~$k$th collisions, during which the configuration has changed by a distance~$<\ve$, from~$K_{k-1}$ to~$K_k$, and the compositions $F_{K_k,K_{k-1}}\circ\dots\circ F_{K_1,K_0}$ represent the dynamics for scatterers moving with speed~$<\ve$. 

In the limit of infinitesimally slowly moving scatterers, we obtain a quasistatic dynamical system: We fix a (piecewise) continuous curve 
\beqn
\gamma:[0,1]\to\cK
\eeqn
and a triangular array of configurations
\beqn
\bfT = \{K_{n,k}\in\cK\ :\ 0\le k\le n, \ n\ge 1\} 
\eeqn
%with $d(K_{n,k-1},K_{n,k})<\ve_n$, where $\lim_{n\to\infty}\ve_n = 0$. Assume moreover that there is a continuous curve $\gamma^\cK:[0,1]\to\cK$ such that 
satisfying 
\beqn
\lim_{n\to\infty}K_{n,\round{nt}} = \gamma_t
\eeqn
uniformly.
We identify the system space~$\cM$ with the metric space~$\cK$ of admissible configurations and write $T_{n,k} = F_{K_{n,k}}$ for~$0\le k\le n$ and~$n\ge 1$. 

In the present quasistatic setting, defining $T_{n,k}$ to be $F_{K_{n,k}}$ instead of $F_{K_{n,k},K_{n,k-1}}$ is a justified simplification: as \mbox{$n\to\infty$}, $d(K_{n,k-1},K_{n,k})$ tends to zero uniformly and the maps $F_{K_{n,k},K_{n,k-1}}$ and $F_{K_{n,k}}$ become indistinguishable.

It is well known that all the billiard maps $T_K\in\cM$ ($K\in\cK$) preserve the same measure, which we denote by~$\hat\mu$. (This is a direct consequence of the fact that the underlying flow is Hamiltonian.) In particular, 
\beqn
\hat\mu_t = \hat\mu
\eeqn
and
\beqn
\zeta(t) = \int_0^t \hat\mu_s(f) \,\rd s = t\,\hat\mu(f)
\eeqn 
for all $t\in[0,1]$ and all $f$.

\medskip
\begin{lem}\label{lem:QDS_billiards}
Let $f$ be Lipschitz continuous and $\mu = \hat\mu$. There exist $C>0$ and $\vartheta\in(0,1)$ such that the correlation functions defined in~\eqref{eq:corr} satisfy
\beqn
|c^{\ell,j}_n(k_1,\dots,k_\ell)| \le C\vartheta^{k_{j+1}-k_j}
\eeqn
for all integers $2\le \ell \le 4$, $1\le j<\ell$ and $0\le k_1\le\dots\le k_\ell$. 
\end{lem}

\medskip
\begin{proof}
Since~$\ve>0$ is small and $d(K_{n,k-1},K_{n,k})<\ve$ for all $k$ and $n$, the theory of~\cite{StenlundYoungZhang_2013} for billiards with slowly moving scatterers applies. Since all of the maps $T_{n,k}$ preserve the same measure $\hat\mu$, the correlation bound is then obtained exactly as in~\cite{Stenlund_2012}.
\end{proof}

\medskip
Remark~\ref{rem:Lip} on the regularity of~$f$ applies also here.

\medskip
We are now in position to prove the following ergodic theorem for quasistatic billiards:

\medskip
\begin{thm}\label{thm:QDS_billiards}
Let $f$ be continuous. Then
\beq\label{eq:uniform_conv_billiards}
\lim_{n\to\infty}\sup_{t\in[0,1]}|\zeta_n(x,t)- t\,\hat\mu(f)| = 0
\eeq
for almost every $x$ in the sense of Lebesgue. In particular, the constant family $(\hat\mu)_{t\in[0,1]}$ is the unique physical family of measures for the quasistatic dynamical system in question.
\end{thm}

\medskip
\begin{proof}
Let $f$ be a Lipschitz continuous function first.
Since $\hat\mu$ is a common invariant measure for the maps, we have $\hat\mu(f_{n,\round{nt}})=\hat\mu(f)=\hat\mu_t(f)$, so conditions~(B1) and~(B2) of Theorem~\ref{thm:main2} are satisfied. On the other hand, Lemma~\ref{lem:QDS_billiards} implies that also~(B3) holds, and so,~\eqref{eq:uniform_conv_billiards} follows for Lipschitz continuous functions and almost every~$x$ with respect to the measure~$\hat\mu$. But~$\hat\mu$ is equivalent to the Lebesgue measure~$\fm$, so we can replace the former by the latter. (Precisely, the density of $\hat\mu$ vanishes on the set of those initial conditions that correspond to a tangential collision with a scatterer, which has zero Lebesgue measure). An approximation argument similar to the one carried out in the proof of Theorem~\ref{thm:QDS_expanding} shows that the result extends to continuous functions. Clearly $\mathscr{P} = (\hat\mu)_{t\in[0,1]}$ is a physical family of measures and the preceding result shows that almost every point with respect to~$\fm$ is $\mathscr{P}$-generic. The uniqueness of~$\mathscr{P}$ now follows from Corollary~\ref{cor:unique}.
\end{proof}

%%%%%%%%%%%%%%%%%%%%%%%%%%%%%%%%%%%%%%%

\medskip
\section{Preliminary lemmas}\label{sec:preliminary}
In this section we lay the foundations for the proofs of Theorems~\ref{thm:main1} and~\ref{thm:main2} by proving two useful lemmas.

\medskip
\begin{lem}\label{lem:Lip_extension_limit}
Let $(\Omega,\cF,\bP)$ be a probability space and $g_n:\Omega\times[0,1]\to\bR$, $n\ge 1$, given functions. Assume that $g_n(\omega,\slot)$ are uniformly Lipschitz continuous. Assume, moreover, that there exist a dense set $A\subset[0,1]$ and a function~$g:\Omega\times A\to \bR$ such that, given~$t\in A$,
\beqn
\lim_{n\to\infty} g_n(\omega,t) = g(\omega,t)
\eeqn
for almost every~$\omega$. (The associated set of full measure may depend on~$t$.)
Then there exists~$\Omega_0\in\cF$ with~$\bP(\Omega_0)=1$ such that, for every~$\omega\in\Omega_0$, $g(\omega,\slot)$ extends to a Lipschitz continuous function on~$[0,1]$, which we still denote by~$g(\omega,\slot)$. Moreover,
\beqn%\label{eq:uniform_conv}
\lim_{n\to\infty}\sup_{t\in[0,1]}|g_n(\omega,t)-g(\omega,t)| = 0
\eeqn
for all $\omega\in\Omega_0$.
\end{lem}

\medskip
In the special case when the functions only depend on~$t$, we readily get the following corollary:

\medskip
\begin{cor}\label{cor:Lip_extension_limit}
Let the functions $g_n:[0,1]\to\bR$, $n\ge 1$, be uniformly Lipschitz continuous. Assume, moreover, that there exist a dense set $A\subset[0,1]$ and a function~$g:A\to \bR$ such that
\beqn
\lim_{n\to\infty} g_n(t) = g(t), \quad t\in A. 
\eeqn
Then $g$ extends to a Lipschitz continuous function on~$[0,1]$, which we still denote by~$g$, with the property that
\beqn%\label{eq:uniform_conv}
\lim_{n\to\infty}\sup_{t\in[0,1]}|g_n(t)-g(t)| = 0.
\eeqn
\end{cor}

\medskip
\begin{proof}[Proof of Lemma~\ref{lem:Lip_extension_limit}]
Denote the supremum of the Lipschitz constants of~$g_n(\omega,\slot)$ by~$L$.
Fix $\ve>0$ arbitrarily. Fix a finite set $\cT_\ve = (t_k)_{k=0}^K$ such that $t_0 = 0$; $t_K=1$; $0<t_{k+1}-t_k \le \frac14 \ve L^{-1}$ for $0\le k< K$; and $t_k\in A$ for $0<k<K$.
Then
\beqn
\sup_{t\in[0,1]}|g_n(\omega,t)-g_m(\omega,t)| \le \max_{0< k<K} |g_n(\omega,t_k)-g_m(\omega,t_k)| + \frac\ve2.
\eeqn
By the convergence assumption, there exists a set $\Omega_\ve\subset\Omega$ with $\bP(\Omega_\ve)=1$ such that, for every $\omega\in\Omega_\ve$, there exists $N_\ve(\omega)\in\bN$ satisfying
\beqn
\max_{0< k<K} |g_n(\omega,t_k)-g(\omega,t_k)| \le \frac\ve4, \quad n\ge N_\ve(\omega).
\eeqn
In particular, 
\beqn
\sup_{t\in[0,1]}|g_n(\omega,t)-g_m(\omega,t)| \le \ve, \quad n,m\ge N_\ve(\omega).
\eeqn
Now, set $\Omega_0 = \cap_{j=1}^\infty \Omega_{j^{-1}}$. Note that $\bP(\Omega_0) = 1$ and $\Omega_0\subset\Omega_\ve$ for all~$\ve>0$. We now see that $(g_n(\omega,\slot))_{n\ge 0}$ is a Cauchy sequence in $C([0,1])$, for all $\omega\in\Omega_0$. As such, it converges uniformly, necessarily to an extension of~$g(\omega,\slot)$. The latter inherits the bound~$L$ on its Lipschitz constant.
\end{proof}

%%%%%%%%%%%%%%%%%%%%%%%%%%%%%

\medskip
The following lemma in probability theory gives a moment condition for almost sure convergence of a sequence of random variables.

\medskip
\begin{lem}\label{lem:as_conv}
Let $(X_n)_{n\ge 1}$ be a sequence of real-valued random variables.
If there exists a number~$p\ge 1$ such that
\beqn
\sum_{n=1}^\infty \bE(|X_n|^p) < \infty,
\eeqn
then
\beqn
\lim_{n\to\infty} X_n = 0
\eeqn
almost surely.
\end{lem}

\medskip
\begin{proof}
Fix~$\ve>0$ arbitrarily. Since~$p\ge 1$,
\beqn
\bE(|X_n|\ge\ve) = \int 1_{\{|X_n|\ge\ve\}}\,\rd\bP \le \int \frac{|X_n|^p}{\ve^p} 1_{\{|X_n|\ge\ve\}}\,\rd\bP \le \ve^{-p}\, \bE(|X_n|^p).
\eeqn
Thus, by assumption,
\beqn
\sum_{n=1}^\infty \bE(|X_n|\ge\ve) \le \ve^{-p} \sum_{n=1}^\infty \bE(|X_n|^p) < \infty.
\eeqn
The Borel--Cantelli lemma now implies that $|X_n|<\ve$ for all large enough~$n$, almost surely. Since~$\ve>0$ was arbitrary,~$X_n$ converges to zero almost surely.
\end{proof}

\medskip
\begin{remark}\label{rem:as_conv}
In a typical application of Lemma~\ref{lem:as_conv}, $X_n = n^{-1}(Y_1+\dots+Y_n)$, where~$(Y_n)_{n\ge 1}$ is a sequence of integrable, centered, random variables, and~$p>2$ is necessary for the summability condition to be satisfied; for instance in the case of independent and identically distributed~$Y_n$, we have $\bE(|X_n|^2) = n^{-1} E(|Y_1|^2)$, which is not summable. When the random variables are not independent, it tends to be easier to bound absolute moments of even order, which in practice leads to~$p\ge 4$. 
\end{remark}

%%%%%%%%%%%%%%%%%%%%%%%%%%%%%%%%%

%%%%%%%%%%%%%%%%%%%%%%%%%%%%%%%%%
%%%%%%%%%%%%%%%%%%%%%%%%%%%%%%%%%

%\pagebreak
\medskip
\section{Proofs of Theorems~\ref{thm:main1}, \ref{thm:main2} and \ref{thm:unique}}\label{sec:proofs}

Here we prove Theorems~\ref{thm:main1} and~\ref{thm:main2} with the aid of Lemma~\ref{lem:fourth} and the preliminary results in Section~\ref{sec:preliminary}.

%%%%%%%%%%%%%%%%%%%%%%%%%%%%%%%%%%%

\medskip
\subsection{Proof of Theorem~\ref{thm:main1}}
Because~$f$ is bounded, the sequence $(t\mapsto\mu(\zeta_n(\slot,t)))_{n\ge 1}$ is uniformly Lipschitz continuous. By the convergence assumption~(A1) and Corollary~\ref{cor:Lip_extension_limit}, there is a Lipschitz extension of the function~$g$ such that
\beqn
\lim_{t\to\infty}\sup_{t\in[0,1]}|\mu(\zeta_n(\slot,t)) - g(t)| = 0.
\eeqn 
It remains to show that
\beqn
\lim_{n\to\infty}\sup_{t\in[0,1]}|\bar\zeta_n(x,t)| = 0
\eeqn
for almost every $x$ with respect to~$\mu$. Given~$x$, also the sequence $(t\mapsto\bar\zeta_n(x,t))_{n\ge 1}$ is uniformly Lipschitz continuous. By Lemma~\ref{lem:Lip_extension_limit}, it suffices to show that, given~$t\in B$,
\beqn%\label{eq:uniform_conv}
\lim_{n\to\infty}\bar\zeta_n(x,t) = 0
\eeqn
for almost every $x$ with respect to~$\mu$. To this end, we recall the summability assumption~(A2) and apply Lemma~\ref{lem:as_conv}. The proof of Theorem~\ref{thm:main1} is now complete.
\qed

%%%%%%%%%%%%%%%%%%%%%%%%%%%%%%%%%%%%%

\medskip
\subsection{Proof of Theorem~\ref{thm:main2}}
The proof builds on the following lemma. Imposing a condition on the fourth moment here (as well as in Theorem~\ref{thm:main2}) reflects the observation made in Remark~\ref{rem:as_conv}.

\medskip
\begin{lem}\label{lem:fourth}
Let~$f$ be a bounded function and~$\mu$ a measure. Under condition~(B3) of Theorem~\ref{thm:main2},
\beqn
\sum_{n=1}^\infty \mu(|\bar\zeta_n(\slot,t)|^4) < \infty
\eeqn
for all $t\in[0,1]$.
\end{lem}

\medskip
\begin{proof}
For brevity, we write $\ave{t} = \mu(f_{n,\round{nt}})$, $\ave{\bar t} = \mu(\bar f_{n,\round{nt}})$, $\ave{t\bar s} = \mu(f_{n,\round{nt}}\bar f_{n,\round{ns}})$, etc. For any $t_1,\dots,t_4\in[0,1]$, we have
\beqn
\begin{split}
\ave{\bar t_1 \bar t_2 \bar t_3 \bar t_4} 
& = \ave{t_1 \bar t_2 \bar t_3 \bar t_4} - \ave{t_1} \ave{\bar t_2 \bar t_3 \bar t_4}
\\
& = \ave{t_1 t_2 \bar t_3 \bar t_4} - \ave{t_2} \ave{t_1 \bar t_3 \bar t_4} - \ave{t_1} \ave{t_2 \bar t_3 \bar t_4} + \ave{t_1} \ave{t_2} \ave{\bar t_3 \bar t_4}
\\
& = \{\ave{t_1 t_2 t_3 t_4} - \ave{t_1 t_2 t_3} \ave{t_4} \}
- \ave{t_3} \{\ave{t_1 t_2 t_4} - \ave{t_1 t_2}\ave{t_4} \}
\\ 
& \qquad 
- \ave{t_2} \{\ave{t_1 t_3 t_4} - \ave{t_1 t_3}\ave{t_4} \}
+ \ave{t_2} \ave{t_3}  \{ \ave{t_1 t_4} - \ave{t_1} \ave{t_4} \}
\\ 
& \qquad 
- \ave{t_1} \{\ave{t_2 t_3 t_4} - \ave{t_2 t_3} \ave{t_4} \}
+ \ave{t_1} \ave{t_3}  \{\ave{t_2 t_4} - \ave{t_2} \ave{t_4} \}
\\ 
& \qquad 
+ \ave{t_1} \ave{t_2} \{\ave{t_3 t_4} - \ave{t_3}  \ave{t_4} \} .
\end{split}
\eeqn
Together with (B3), this shows that
\beq\label{eq:quad1}
|\ave{\bar t_1 \bar t_2 \bar t_3 \bar t_4}| \le C\Phi(\round{nt_4}-\round{nt_3}), \quad t_1\le t_2\le t_3\le t_4.
\eeq
It also shows that
\beqn
|\ave{\bar t_1 \bar t_2 \bar t_3 \bar t_4}| \le C\Phi(\round{nt_3}-\round{nt_4}), \quad t_1\ge t_2\ge t_3\ge t_4,
\eeqn
which is the same as
\beq\label{eq:quad2}
|\ave{\bar t_1 \bar t_2 \bar t_3 \bar t_4}| \le C\Phi(\round{nt_2}-\round{nt_1}), \quad t_1\le t_2\le t_3\le t_4.
\eeq
Combining~\eqref{eq:quad1} and~\eqref{eq:quad2} and reverting to original notation, we obtain
\beqn
|\mu(\bar f_{n,\round{nt_1}}\cdots \bar f_{n,\round{nt_4}})| \le C\min(\Phi(\round{nt_2}-\round{nt_1}),\Phi(\round{nt_4}-\round{nt_3})), \quad 0\le t_1\le t_2\le t_3\le t_4.
\eeqn

Next, using that the value of $\mu(\bar f_{n,\round{nt_1}}\cdots \bar f_{n,\round{nt_4}})$ remains unchanged under permutations of~$t_1,\dots,t_4$, and noting that $\min(a,b)\le a^\frac12b^\frac12$ for nonnegative~$a$ and~$b$, we see that
\beqn
\begin{split}
\mu(|\bar\zeta_n(\slot,t)|^4) & = \int_0^t \! \int_0^t \! \int_0^t \!  \int_0^t \mu(\bar f_{n,\round{nt_1}}\cdots \bar f_{n,\round{nt_4}})\,\rd t_1\,\rd t_2\, \rd t_3\, \rd t_4
\\
& = 4! \int_0^t \! \int_0^{t_4} \!\!\! \int_0^{t_3} \!\!\! \int_0^{t_2} \mu(\bar f_{n,\round{nt_1}}\cdots \bar f_{n,\round{nt_4}})\,\rd t_1\,\rd t_2\, \rd t_3\, \rd t_4
\\
& \le C \int_0^1 \! \int_0^{t_4} \!\!\! \int_0^{t_3} \!\!\! \int_0^{t_2} \min(\Phi(\round{nt_2}-\round{nt_1}),\Phi(\round{nt_4}-\round{nt_3})) \,\rd t_1\,\rd t_2\, \rd t_3\, \rd t_4
\\
& \le C \int_0^1 \!  \int_0^{t_4} \!\!\! \int_0^{1} \! \int_0^{t_2} \Phi(\round{nt_2}-\round{nt_1})^{\frac12}\Phi(\round{nt_4}-\round{nt_3})^{\frac12} \,\rd t_1\,\rd t_2\, \rd t_3\, \rd t_4
\\
& = C \int_0^1 \! \int_0^{t_4} \Phi(\round{nt_4}-\round{nt_3})^{\frac12} \, \rd t_3\, \rd t_4\, \int_0^{1} \! \int_0^{t_2} \Phi(\round{nt_2}-\round{nt_1})^{\frac12} \,\rd t_1\,\rd t_2.
%\\
%& = C \left(\,\int_0^1 \! \int_0^s \Phi(\round{ns}-\round{nr})^{\frac12} \,\rd r\,\rd s\right)^2.
\end{split}
\eeqn
We have thus shown
\beqn
\sup_{t\in[0,1]}\mu(|\bar\zeta_n(\slot,t)|^4) \le C \left(\,\int_0^1 \! \int_0^s \Phi(\round{ns}-\round{nr})^{\frac12} \,\rd r\,\rd s\right)^2.
\eeqn
But
\beqn
\begin{split}
 \int_0^1 \! \int_0^s \Phi(\round{ns}-\round{nr})^{\frac12} \,\rd r\,\rd s
%\\
%& = \frac1n \int_0^{n} \! \int_0^{s/n} (\Phi(\round{s}-\round{nr}))^{\frac12} \,\rd r\,\rd s
& = n^{-2} \int_0^n \! \int_0^s \Phi(\round{s}-\round{r})^{\frac12} \,\rd r\,\rd s
%\\
%& = n^{-2} \int_1^n \! \int_0^{s-1} \Phi(\round{s}-\round{r})^{\frac12} \,\rd r\,\rd s + O(n^{-1})
%\\
%& = n^{-2} \int_1^n \! \int_0^{s-1} \Phi(s-r-1)^{\frac12} \,\rd r\,\rd s + O(n^{-1})
\\
& = n^{-2} \int_0^n \! \int_0^{s} \Phi(r)^{\frac12} \,\rd r\,\rd s + O(n^{-1})
\\
& = n^{-2} \int_0^n (n-s) \Phi(s)^{\frac12} \,\rd s + O(n^{-1})
%\\
%& = n^{-2} \int_1^n \! \int_0^{s} \Phi(r)^{\frac12} \,\rd r\,\rd s + O(n^{-1})
%\\
%& = n^{-2} \int_0^n \! \int_0^{s} \Phi(r)^{\frac12} \,\rd r\,\rd s + O(n^{-1})
%\\
%& = n^{-2} \int_0^n (n-s) \Phi(s)^{\frac12} \,\rd s + O(n^{-1})
\\
& \le n^{-1} \int_0^n \Phi(s)^{\frac12} \,\rd s + O(n^{-1})
\\
& = O((n^\frac12\log n)^{-1}) .
\end{split}
\eeqn
The last line above is due to the fact that the function $\Psi(s) = s^{\frac12}(\log s)^{-1}$ satisfies
\beqn
\frac{d}{ds}\Psi(s) = \frac{\log s-2}{2s^{\frac12}(\log s)^2} \ge \frac{1}{2s^{\frac12}\log s} = \tfrac12\Phi(s)^\frac12, \quad s\ge 2,
\eeqn
so
\beqn
\int_0^n \Phi(s)^{\frac12} \,\rd s
\le 2\Phi(0)^\frac12 + 2\Psi(n) - 2\Psi(2) = O(n^\frac12(\log n)^{-1}).
\eeqn
Hence,
\beqn
\sup_{t\in[0,1]}\mu(|\bar\zeta_n(\slot,t)|^4) \le \frac{C}{n(\log n)^2}, \quad n\ge 2.
\eeqn
Observing finally that $\bar\zeta_1(x,t) = t(f(x)-\mu(f))$ implies
\beqn
\sup_{t\in[0,1]}\mu(|\bar\zeta_1(\slot,t)|^4) \le (2\|f\|_\infty)^4,
\eeqn
the proof is complete.
\end{proof}

\medskip
\begin{remark}
It may be of some interest to note that the derivation of the upper bound $\int_0^1 \! \int_0^s \Phi(\round{ns}-\round{nr})^{\frac12} \,\rd r\,\rd s = O((n^\frac12 \log n)^{-1})$ in the proof of Lemma~\ref{lem:fourth} appears somewhat crude, but it gives the right order of magnitude. Indeed, fixing a small $\ve>0$, we also get the lower bound 
\beqn
n^{-2} \int_0^n (n-s) \Phi(s)^{\frac12} \,\rd s 
> n^{-2} \int_{(1-\ve)n}^n (n-s) \Phi(s)^{\frac12} \,\rd s
= \int_{(1-\ve)}^1 (1-s) \Phi(ns)^{\frac12} \,\rd s,
\eeqn
where $\Phi(ns)^{\frac12} = (n^\frac12 \log n)^{-1} (1+O(\ve/\log n)) s^{-\frac12}$ for $s\in[1-\ve,1]$.
\end{remark}

\medskip
Let us proceed with the proof of Theorem~\ref{thm:main2}. By Lemma~\ref{lem:fourth}, (B3)~of Theorem~\ref{thm:main2} implies~(A2) of Theorem~\ref{thm:main1}. On the other hand, since~(B1) and~(B2) hold, the bounded convergence theorem yields~(A1) with~$g = \zeta$. Hence, the conditions of Theorem~\ref{thm:main1} are satisfied, and we conclude that
\beqn
\lim_{n\to\infty}\sup_{t\in[0,1]}|\zeta_n(x,t)-\zeta(t)| = 0
\eeqn
for almost every $x$ with respect to $\mu$. This finishes the proof of Theorem~\ref{thm:main2}.
\qed

%%%%%%%%%%%%%%%%%%%%%%%%%%%%%%%%%%%%%

\medskip
\subsection{Proof of Theorem~\ref{thm:unique}}\label{sec:unique_proof}
Write $\mathscr{P} = (\hat\mu_t)_{t\in[0,1]}$ and $\mathscr{P}' = (\hat\mu'_t)_{t\in[0,1]}$. 
 By the simultaneous genericity of some point~$x$ for both physical families of measures, given a bounded continuous~$f$,~\eqref{eq:physical_QDS} implies
\beqn
\int_a^b \hat\mu_s(f)\,\rd s = \int_a^b \hat\mu_s'(f)\,\rd s, \quad a,b\in[0,1].
\eeqn
By the Lebesgue differentiation theorem, there exists a Borel set $B_f \subset [0,1]$ of Lebesgue measure~$1$ such that $\hat\mu_t(f) = \hat\mu_t'(f)$ for all $t\in B_f$. Let $\{U_i:i\in\bN\}$ be a countable base for the topology. Given $i\in\bN$, there exists a sequence of bounded continuous functions $f_{i,n}:X\to\bR$, $n\ge 1$, such that~$f_{i,n}$ converges pointwise to the indicator function of~$U_i$ as~$n\to\infty$. Then the bounded convergence theorem yields $\hat\mu_t(f_{i,n})\to\hat\mu_t(U_i)$ and~$\hat\mu_t'(f_{i,n})\to\hat\mu_t'(U_i)$ for all $t\in[0,1]$, so
\beq\label{eq:agree_on_basis}
\hat\mu_t(U_i) = \hat\mu_t'(U_i)
\eeq
for all $t\in B_i = \cap_{n\ge 1} B_{f_{i,n}}$ and $i\in\bN$. Thus, for $t\in B = \cap_{i\in\bN} B_i$ the measures~$\hat\mu_t$ and~$\hat\mu_t'$ coincide, because they agree on the base. Since each~$B_i$ has Lebesgue measure~$1$, the same is true of~$B$.
\qed

%%%%%%%%%%%%%%%%%%%%%%%%%%%%%%%%%%%%
%%%%%%%%%%%%%%    References    %%%%%%%%%%%%%
%%%%%%%%%%%%%%%%%%%%%%%%%%%%%%%%%%%%

%\vskip 1cm
\medskip

%\newpage

\bigskip
\bigskip
\bibliography{Quasistatic}{}

\begin{thebibliography}{10}

\bibitem{Aimino_etal_2014}
Romain Aimino, Huyi Hu, Matt Nicol, Andrew T\"or\"ok, and Vaienti Sandro.
\newblock Polynomial loss of memory for maps of the interval with a neutral
  fixed point.
\newblock arXiv:1402.4399.
\newblock Available from: \url{http://arxiv.org/abs/1402.4399}.

\bibitem{ChernovMarkarian_2006}
Nikolai Chernov and Roberto Markarian.
\newblock {\em Chaotic billiards}, volume 127 of {\em Mathematical Surveys and
  Monographs}.
\newblock American Mathematical Society, Providence, RI, 2006.

\bibitem{DobbsStenlund_2014}
Neil Dobbs and Mikko Stenlund.
\newblock Quasistatic dynamical systems.
\newblock Submitted.
\newblock Available from: \url{http://arxiv.org/abs/1504.01926}.

\bibitem{GuptaOttTorok_2013}
Chinmaya Gupta, William Ott, and Andrei T{\"o}r{\"o}k.
\newblock Memory loss for time-dependent piecewise expanding systems in higher
  dimension.
\newblock {\em Math. Res. Lett.}, 20(1):141--161, 2013.
\newblock Available from: \url{http://dx.doi.org/10.4310/MRL.2013.v20.n1.a12},
  \href {http://dx.doi.org/10.4310/MRL.2013.v20.n1.a12}
  {\path{doi:10.4310/MRL.2013.v20.n1.a12}}.

\bibitem{MohapatraOtt_2014}
Anushaya Mohapatra and William Ott.
\newblock Memory loss for nonequilibrium open dynamical systems.
\newblock {\em Discrete Contin. Dyn. Syst.}, 34(9):3747--3759, 2014.
\newblock Available from: \url{http://dx.doi.org/10.3934/dcds.2014.34.3747},
  \href {http://dx.doi.org/10.3934/dcds.2014.34.3747}
  {\path{doi:10.3934/dcds.2014.34.3747}}.

\bibitem{OttStenlundYoung_2009}
William Ott, Mikko Stenlund, and Lai-Sang Young.
\newblock Memory loss for time-dependent dynamical systems.
\newblock {\em Mathematical Research Letters}, 16(3):463--475, 2009.
\newblock Available from:
  \url{http://www.intlpress.com/_newsite/site/pub/pages/journals/items/mrl/con%
tent/vols/0016/0003/00020435/index.php}.

\bibitem{Stenlund_2011}
Mikko Stenlund.
\newblock Non-stationary compositions of {A}nosov diffeomorphisms.
\newblock {\em Nonlinearity}, 24:2991--3018, 2011.
\newblock \href {http://dx.doi.org/doi:10.1088/0951-7715/24/10/016}
  {\path{doi:doi:10.1088/0951-7715/24/10/016}}.

\bibitem{Stenlund_2012}
Mikko Stenlund.
\newblock {\em Commun. Math. Phys.}, 325:879--916, 2014.
\newblock Available from: \url{http://dx.doi.org/10.1007/s00220-013-1870-3}.

\bibitem{StenlundYoungZhang_2013}
Mikko Stenlund, Lai-Sang Young, and Hongkun Zhang.
\newblock Dispersing billiards with moving scatterers.
\newblock {\em Commun. Math. Phys.}, 322(3):909--955, 2013.
\newblock Available from: \url{http://dx.doi.org/10.1007/s00220-013-1746-6}.

\bibitem{Young_2002}
Lai-Sang Young.
\newblock What are {SRB} measures, and which dynamical systems have them?
\newblock {\em J. Statist. Phys.}, 108(5-6):733--754, 2002.
\newblock Available from: \url{http://dx.doi.org/10.1023/A:1019762724717}.

\end{thebibliography}
\bibliographystyle{plainurl}

%%%%%%%%%%%%%%%%%%%%%%%%%%%%%%%%%%%%

\vspace*{\fill}

\end{document}